\documentclass[11pt]{article}
\usepackage[T1]{fontenc}
\usepackage[utf8]{inputenc}
\usepackage[english]{babel}
\usepackage{amsmath,amssymb,amsthm}
\usepackage[a4paper,margin=2.5cm]{geometry}
\usepackage[unicode]{hyperref}
\usepackage{enumitem}

\newtheorem{theorem}{Theorem}
\newtheorem{lemma}{Lemma}
\newtheorem{proposition}{Proposition}
\newtheorem{fact}{Fact}
\newtheorem{cor}{Corollary}
\theoremstyle{definition}

\newtheorem{example}{Example}
\newtheorem{question}{Question}
\theoremstyle{remark}
\newtheorem{remark}{Remark}

\newcommand{\R}{\mathbb R}
\newcommand{\Ls}{L'_\sigma}
\newcommand{\spn}{\operatorname{span}}

\title{Locally convex spaces\\ with hemicompact weak$^*$ dual}
\author{Evgenii Reznichenko\thanks{Department of General Topology and
Geometry, Faculty of Mechanics and Mathematics, M.\,V.~Lomonosov Moscow
State University, Leninskie Gory~1, Moscow, 119991, Russia; Shenzhen
MSU-BIT University, Shenzhen, China. E-mail: \texttt{erezn@inbox.ru}}}
\date{}

\begin{document}
\maketitle

\begin{abstract}
For a locally convex space $L$ let $\Ls$ denote its dual endowed with the
weak$^*$ topology $\sigma(L',L)$. We prove that $\Ls$ is hemicompact if
and only if $L$ with the Mackey topology $\tau(L,L')$ is metrizable and
barrelled, and if and only if $\Ls$ is $\sigma$-compact and
quasi-complete. The key step is that a hemicompact $\Ls$ is
quasi-complete. The Mackey topology $\tau(L,L')$ is metrizable if and
only if $L'$ is a countable union of absolutely convex $w^*$-compact sets;
spaces with normable barrelled Mackey topology are characterized
similarly. As a consequence, a metrizable (in particular, normed) space
$L$ is barrelled if and only if $\Ls$ is hemicompact. For function spaces
we obtain: $(C_k(X))'_\sigma$ is hemicompact if and only if $X$ is
hemicompact, and $(C_p(X))'_\sigma$ is hemicompact if and only if $X$ is
countable and all compact subsets of $X$ are finite.
\end{abstract}



\section{Introduction}

Let $(K_n)_{n\in\mathbb N}$ be a sequence of compact subsets of a
topological space $X$. The space $X$ is said to be \emph{hemicompact with
respect to} $(K_n)$ if every compact set $K\subseteq X$ is contained in
some $K_n$; $X$ is \emph{hemicompact} if it is hemicompact with respect
to some sequence of compact sets. A Tychonoff space $X$ is hemicompact if
and only if the space $C_k(X)$ of continuous functions on $X$ with the
compact-open topology is metrizable (Arens~\cite{Arens}; see also
\cite[Theorem~4.4.2 and the remark after it]{McCoyNtantu}, \cite[Theorem~5.8.5]{NB}, \cite[10.2.1]{PCB}).

Let $L$ be a Banach space. By the Banach--Alaoglu theorem the balls
$nB_{L'}$ are compact in $\Ls$, and by the Banach--Steinhaus theorem
$\sigma(L',L)$-bounded (in particular, compact) subsets of $L'$ are norm
bounded. Hence $\Ls$ is hemicompact with respect to the sequence of balls
$(nB_{L'})$. This is a folklore fact: it is stated without proof in
\cite[Example~1.3(c)]{KampoukosMercourakis2013} and proved in
\cite[Proposition~5.6]{LRS-TVS} (see also Remark~5.7 there).

For a normed space $L$ the property ``every $\sigma(L',L)$-bounded subset
of $L'$ is norm bounded'' characterizes barrelledness of $L$: this
follows from the fact that the equicontinuous subsets of $L'$ are
precisely the norm bounded ones and from
\cite[Ch.~IV, \S\,2, Theorem~1]{Bourbaki} together with the bipolar
theorem.\footnote{By Webb and
Saxon--Wilansky, it even suffices to consider $w^*$-convergent sequences
\cite[Proposition~4.1]{Webb}, \cite[Theorem~1.5]{SW}.} Therefore, if $L$
is a barrelled normed space, then $\Ls$ is hemicompact (with respect to
the balls $nB_{L'}$). In fact, hemicompactness of $\Ls$
characterizes barrelledness of $L$.

\begin{theorem}\label{thm:normed}
For a normed space $L$ the following are equivalent:
\begin{enumerate}[label=\textup{(\alph*)}]
\item $L$ is barrelled;
\item $\Ls$ is hemicompact with respect to the sequence of balls
$(nB_{L'})$;
\item $\Ls$ is hemicompact;
\item every compact subset of $\Ls$ is norm bounded;
\item every $\sigma(L',L)$-null sequence in $L'$ is norm bounded.
\end{enumerate}
\end{theorem}

Thus, for normed spaces, hemicompactness of $\Ls$, whenever it holds, is
realized by the balls. This result extends to metrizable spaces.

\begin{theorem}\label{thm:metr}
Let $L$ be a metrizable locally convex space and
$U_1\supseteq U_2\supseteq\dots$ a base of absolutely convex closed
neighbourhoods of zero. The following are equivalent:
\begin{enumerate}[label=\textup{(\alph*)}]
\item $L$ is barrelled;
\item $\Ls$ is hemicompact with respect to the sequence of polars
$(U_n^\circ)$;
\item $\Ls$ is hemicompact;
\item $\Ls$ is quasi-complete;
\item every $\sigma(L',L)$-null sequence in $L'$ is equicontinuous;
\item every compact subset of $\Ls$ is $\beta(L',L)$-bounded, i.e., $L$
is g-barrelled in the sense of~\cite{CMT99} (every compact subset of
$\Ls$ is equicontinuous).
\end{enumerate}
\end{theorem}

If $\Ls$ is hemicompact, then, as we shall see, the Mackey topology
$\tau(L,L')$ is metrizable. Therefore, for an arbitrary $L$ it is natural
to determine first when the Mackey topology is metrizable.

\begin{proposition}\label{prop:mackey}
For a locally convex space $L$ the following are equivalent:
\begin{enumerate}[label=\textup{(\arabic*)}]
\item the Mackey topology $\tau(L,L')$ is metrizable;
\item there is a metrizable locally convex topology on $L$ compatible
with the duality $\langle L,L'\rangle$;
\item $L'$ is the union of countably many absolutely convex compact
subsets of $\Ls$;
\item the family of absolutely convex compact subsets of $\Ls$ has a
countable cofinal subfamily $(A_n)$: every absolutely convex compact set
is contained in some $A_n$.
\end{enumerate}
\end{proposition}

However, metrizability of $\tau(L,L')$ is not sufficient for
hemicompactness of $\Ls$ (Example~\ref{ex:phi}); barrelledness is also
needed. We obtain the following complete characterization.

\begin{theorem}\label{thm:main}
For a locally convex space $L$ the following are equivalent:
\begin{enumerate}[label=\textup{(\roman*)}]
\item $\Ls$ is hemicompact;
\item $(L,\tau(L,L'))$ is metrizable and barrelled;
\item $\beta(L,L')=\tau(L,L')$ and this topology is metrizable;
\item $\Ls$ is quasi-complete and has a countable fundamental system of
bounded sets;
\item $\Ls$ is $\sigma$-compact and quasi-complete;
\item $L'$ is the union of countably many absolutely convex compact
subsets of $\Ls$, and $(L,\tau(L,L'))$ is barrelled.
\end{enumerate}
\end{theorem}

Theorem~\ref{thm:main} and Proposition~\ref{prop:mackey} yield
\begin{equation}\label{eq:chain}
\Ls\ \text{is hemicompact}\ \Longrightarrow\ \tau(L,L')\ \text{is
metrizable}\ \Longrightarrow\ \Ls\ \text{is}\ \sigma\text{-compact},
\end{equation}
and for quasi-complete $\Ls$ all three conditions are equivalent. In
general both implications are irreversible, and none of the conditions
``$\tau(L,L')$ is metrizable'', ``$\Ls$ is quasi-complete'', ``$\Ls$ is
$\sigma$-compact'' alone implies hemicompactness of $\Ls$:
\begin{enumerate}[label=\textup{(\alph*)}]
\item there is $L$ with metrizable $\tau(L,L')$ and non-hemicompact
$\Ls$;
\item there is $L$ with $\sigma$-compact but not quasi-complete $\Ls$ and
non-metrizable $\tau(L,L')$;
\item there is $L$ with quasi-complete but not $\sigma$-compact $\Ls$.
\end{enumerate}
For (a) one can take $L=(\varphi,\sigma(\varphi,\varphi))$
(Example~\ref{ex:phi}): $\tau(L,L')$ is metrizable and $\Ls$ is
$\sigma$-compact, but $\Ls$ is not hemicompact and hence, by
\eqref{eq:chain}, not quasi-complete. For (b), take $L=C_p([0,1])$: the
space $\Ls=L_p([0,1])$ is $\sigma$-compact, while $\tau(L,L')$ is not
metrizable (Proposition~\ref{prop:Cp}); by \eqref{eq:chain}, $\Ls$ is not
quasi-complete. For (c), take $\varphi$ with the finest locally convex
topology: $\Ls=\R^{\mathbb N}$ is complete but not $\sigma$-compact, so
by \eqref{eq:chain} $\Ls$ is not hemicompact and $\tau(L,L')$ is not
metrizable. Examples (a) and (b) show that the first and the second
implications in \eqref{eq:chain}, respectively, are irreversible.
Hemicompactness of $\Ls$ may fail also for normed $L$
(Example~\ref{ex:normed}).

Condition (ii) depends only on the duality $\langle L,L'\rangle$: for
instance, for a Banach space $F$ with the weak topology, $L=(F,w)$, the
space $\Ls=F'_{w^*}$ is hemicompact, although $L$ itself is neither
metrizable (if $\dim F=\infty$) nor barrelled.

The main new ingredient of Theorem~\ref{thm:main} appears to be
Lemma~\ref{lem:seq} and its Corollary~\ref{cor:qc}: a hemicompact $\Ls$ is
quasi-complete. Most of the other ingredients are classical. In
Proposition~\ref{prop:mackey} the equivalence (1)$\Leftrightarrow$(4) is
immediate from the description of $\tau(L,L')$ as the topology of uniform
convergence on absolutely convex weakly compact sets, and
(3)$\Rightarrow$(4) is a Baire category argument of the type used in
\cite[Theorem~2.7 and Remark~2.8]{KL22}. Condition (f) of
Theorem~\ref{thm:metr} is equivalent to (e) by the Webb and
Saxon--Wilansky criterion (Fact~\ref{f:known}(5)), so for metrizable
spaces g-barrelledness coincides with barrelledness; we include it for
the application in Remark~\ref{rem:KLq}. Corollary~\ref{cor:Ck} on
$(C_k(X))'_\sigma$ is close to \cite[Theorem~2.11]{KL22}, where it is
shown that $(C_k(X),w)$ is Eberlein--Grothendieck, i.e., $(C_k(X))'_\sigma$
is $\sigma$-compact, if and only if $X$ is hemicompact; our
Lemma~\ref{lem:Clambda} extends \cite[Theorem~2.12 and
Proposition~2.13]{KL22} from $C_k(X)$ to $C_\lambda(X)$, and
Theorem~\ref{thm:normable} is a counterpart of
\cite[Theorem~3.8]{KL22}. We have not found Theorem~\ref{thm:main} or
Proposition~\ref{prop:Lb} in the literature; for other topological
properties of weak and weak$^*$ duals (networks, the Ascoli property,
$k$-properties) see, e.g., \cite{GKKM15,Gab23} and the references
therein.
The implication (v)$\Rightarrow$(ii) for barrelled $L$ is a theorem of
K\c{a}kol and Leiderman \cite[Theorem~2.7]{KL22}; we give its proof in
the form we need (Lemma~\ref{lem:KL}). Recall also that, by
\cite[Proposition~2.1]{KL22}, $\Ls$ is $\sigma$-compact if and only if
$(L,w)$ is an Eberlein--Grothendieck space, i.e., embeds homeomorphically
into $C_p(Q)$ for some compact space $Q$.

By the Arens theorem, condition (i) of Theorem~\ref{thm:main} means that
$C_k(\Ls)$ is metrizable. The map $x\mapsto(f\mapsto f(x))$ linearly
embeds $L$ into $C_k(\Ls)$; the induced topology $\tau_c$ is the topology
of uniform convergence on compact subsets of $\Ls$, and
$\tau(L,L')\subseteq\tau_c\subseteq\beta(L,L')$. If the closed absolutely
convex hull of every compact subset of $\Ls$ is compact (for instance, if
$\Ls$ is quasi-complete), then $\tau_c=\tau(L,L')$. In particular, for
hemicompact $\Ls$ the space $(L,\tau(L,L'))$ is linearly homeomorphic to a
subspace of the metrizable space $C_k(\Ls)$ (Corollary~\ref{cor:qc}). In
general $\tau_c\ne\tau(L,L')$: for $L=C_p([0,1])$ the absolutely convex
compact subsets of $\Ls=L_p([0,1])$ are finite-dimensional
(Lemma~\ref{lem:Lpac}), so $\tau(L,L')$ is the topology of pointwise
convergence, whereas $\tau_c$ dominates uniform convergence on the
compact set $[0,1]\subseteq L_p([0,1])$.

Theorem~\ref{thm:main} describes the locally convex spaces whose Mackey
topology is metrizable and barrelled. The same method describes the
spaces with normable barrelled Mackey topology, i.e., the spaces whose
duality $\langle L,L'\rangle$ is realized by a barrelled normed space.

\begin{theorem}\label{thm:normable}
For a locally convex space $L$ the following are equivalent:
\begin{enumerate}[label=\textup{(\alph*)}]
\item $(L,\tau(L,L'))$ is a normable barrelled space;
\item $\Ls$ is quasi-complete and is the linear span of a compact set;
\item $\Ls$ is quasi-complete and $(L,w)$ linearly embeds into $C_p(Q)$
for some compact space $Q$.
\end{enumerate}
\end{theorem}

In \S\,\ref{sec:func} we study the duals of function spaces: for
$C_k(X)$, hemicompactness of $(C_k(X))'_\sigma$ is equivalent to
hemicompactness of $X$ (Corollary~\ref{cor:Ck}), whereas for $C_p(X)$ it
is required that $X$ be countable and all compact subsets of $X$ be finite
(Proposition~\ref{prop:Cp}).

Hemicompactness of $\Ls$ is never accompanied by the $k$-property, except
in the trivial case: $\Ls$ is a $k_\omega$-space if and only if
$\dim L<\infty$ \cite{R-kd}.

\section{Notation and known facts}

All spaces are real and Hausdorff. For a locally convex space (LCS)
$(L,\tau)$ we denote by $L'$ its dual, $\Ls=(L',\sigma(L',L))$, by $L^*$
the algebraic dual with the topology $\sigma(L^*,L)$; $\tau(L,L')$ and
$\beta(L,L')$ are the Mackey and the strong topologies. For $A\subseteq L$
\[
A^\circ=\{f\in L':\ |f|\le1\ \text{on}\ A\},\qquad
A^{\circ*}=\{\psi\in L^*:\ |\psi|\le1\ \text{on}\ A\},
\]
and for $K\subseteq L'$ we put $K_\circ=\{x\in L:\ |f(x)|\le1\ \forall
f\in K\}$. A \emph{barrel} is a closed absolutely convex absorbing set;
$L$ is \emph{barrelled} if every barrel is a neighbourhood of zero. A set
$A\subseteq L'$ is \emph{equicontinuous} if $A\subseteq U^\circ$ for a
neighbourhood of zero $U$. We call $\Ls$ \emph{quasi-complete} if every
bounded subset of $\Ls$ is relatively compact (for weak topologies this
is equivalent to the usual quasi-completeness, since bounded sets are
precompact there). A sequence $(B_n)$ of bounded subsets of $\Ls$ is
called a \emph{fundamental system of bounded sets} if every bounded subset
of $\Ls$ is contained in some $B_n$.

For a Tychonoff space $X$, $E(X)$ denotes the vector space with basis
$X$, i.e., the finite linear combinations of Dirac measures $\delta_x$,
$x\in X$; we put $L_p(X)=(E(X),\sigma(E(X),C(X)))$ and
$L_b(X)=(E(X),\sigma(E(X),C_b(X)))$. By $\upsilon X$ we denote the Hewitt
realcompactification of $X$; every $f\in C(X)$ extends uniquely to
$f^\upsilon\in C(\upsilon X)$ \cite{Eng}.

\begin{fact}\label{f:known}
\begin{enumerate}[label=\textup{(\arabic*)}]
\item \textup{(Alaoglu--Bourbaki)} For a neighbourhood of zero $U$ the
polar $U^\circ$ is compact in $\Ls$.
\item The closed absolutely convex bounded subsets of $\Ls$ are precisely
the polars $D^\circ$ of barrels $D\subseteq L$, and $(D^\circ)_\circ=D$
\textup{(}bipolar theorem\textup{)}.
\item A topology on $L$ compatible with the duality $\langle
L,L'\rangle$ is barrelled if and only if it coincides with
$\beta(L,L')$.
\item Metrizable LCS are Mackey spaces, i.e., $\tau=\tau(L,L')$
\cite{NB,PCB}.
\item \textup{(Webb \cite[Proposition~4.1]{Webb}; see also
\cite[Theorem~1.5]{SW})} If $L$ is quasi-barrelled (e.g., metrizable)
and every $\sigma(L',L)$-convergent sequence is equicontinuous (i.e., $L$
is $c_0$-barrelled in the sense of \cite[8.2.22]{PCB}), then $L$ is
barrelled.
\item A Mackey space $L$ is barrelled if and only if $\Ls$ is
quasi-complete \cite[Corollary~4.1.15]{PCB}; the implication
``barrelled $\Rightarrow$ $\Ls$ quasi-complete'' (and the equivalence of
weak boundedness, relative weak compactness and equicontinuity in $L'$
for barrelled $L$) is \cite[Ch.~IV, \S\,2, Theorem~1 and
Corollary~2]{Bourbaki}.
\item If $A\subseteq\Ls$ is an absolutely convex compact set, then
$L'_A=\spn A$ with the Minkowski functional of $A$ is a Banach space,
and its norm topology is finer than $\sigma(L',L)|_{L'_A}$
\cite[3.2.2, 3.2.5]{PCB}.
\item The topology $\sigma(L,L')$ is metrizable if and only if
$\dim L'\le\aleph_0$: a countable base of neighbourhoods of zero is given
by a countable set of functionals whose linear span is $L'$
\cite{Jarchow}.
\item Metrizable LCS are quasi-barrelled; for a quasi-barrelled $L$ a
subset of $L'$ is $\beta(L',L)$-bounded if and only if it is
equicontinuous \cite[Definition~4.1.1(ii) and Observation~4.1.2(b)]{PCB}.
\item An LCS of countable dimension is barrelled only in the finest
locally convex topology \cite[Ch.~III, \S\,1, Exercise~4(b)]{Bourbaki}; in
particular, an infinite-dimensional metrizable LCS of countable dimension
is not barrelled.
\item A product of LCS is barrelled if and only if all factors are
barrelled \cite[Proposition~4.2.5(i)]{PCB}.
\item A feral TVS is quasi-complete \cite[Proposition~2.9]{GOR-Baire}.
\end{enumerate}
\end{fact}

\begin{fact}\label{f:func}
Let $X$ be a Tychonoff space.
\begin{enumerate}[label=\textup{(\arabic*)}]
\item $C_p(X)$ is quasi-barrelled and hence a Mackey space
\cite[Corollary~11.7.3]{Jarchow}.
\item (Buchwalter--Schmets \cite{BS73}) $C_p(X)$ is barrelled if and only
if all functionally bounded subsets of $X$ are finite.
\item (Okunev \cite{Ok89}; see also \cite[Corollary~1.4]{KL22}) $L_p(X)$
is $\sigma$-compact if and only if $X$ is $\sigma$-compact.
\item The support of a bounded subset of $L_p(X)$ is functionally bounded
\cite[Lemma~7.5]{LRS-TVS}.
\item (Nachbin--Shirota, see \cite[Theorem~10.1.20]{PCB}) $C_k(X)$ is barrelled if
and only if every closed functionally bounded subset of $X$ is compact.
\item A Lindel\"of (in particular, countable) space is realcompact; in a
realcompact space closed functionally bounded subsets are compact; if
$X\subseteq Z\subseteq\upsilon X$ and $Z$ is realcompact, then
$Z=\upsilon X$ \cite[\S\,3.11]{Eng}.
\end{enumerate}
\end{fact}

\section{Proofs of the theorems}

\begin{lemma}\label{lem:seq}
If $\Ls$ is not quasi-complete, then for every sequence $(K_n)$ of
compact subsets of $\Ls$ there is a compact set of the form
$C=\{0\}\cup\{h_k:k\in\mathbb N\}$, $h_k\to0$ in $\Ls$, not contained in
any $K_n$.
\end{lemma}

\begin{proof}
Replacing $K_n$ by $K_1\cup\dots\cup K_n$, we may assume that $(K_n)$ is
increasing. By Fact~\ref{f:known}(2) there is a barrel $D\subseteq L$
with non-compact polar $D^\circ$. The set $D^{\circ*}$ is pointwise
bounded and closed in $(L^*,\sigma(L^*,L))$, hence compact; therefore
$D^{\circ*}\not\subseteq L'$ (otherwise $D^\circ=D^{\circ*}$ would be
compact).

Since $D$ is absorbing, $D^\circ$ is $\sigma(L',L)$-bounded and therefore
contains no lines; the Minkowski functional $q$ of $D^\circ$ is a norm on
$F=\spn D^\circ$. If $h_k\in F$ and $q(h_k)\to0$, then
$|h_k(x)|\le t\,q(h_k)$ for $x\in tD$, and $D$ is absorbing; hence
$h_k\to0$ in $\Ls$.

Suppose that $\{f\in F:\ q(f)<1/N\}\subseteq K_N$ for some $N$. Then
$\varepsilon D^\circ\subseteq K_N$ for $\varepsilon=1/(2N)$. The set
$K_N$ is also compact in $(L^*,\sigma(L^*,L))$ and hence closed there. By
the bipolar theorem for the pair $(L,L^*)$ and the equality
$(D^\circ)_\circ=D$, the closure of $D^\circ$ in $L^*$ is $D^{\circ*}$.
Hence $\varepsilon D^{\circ*}\subseteq K_N\subseteq L'$, contradicting
$D^{\circ*}\not\subseteq L'$.

Thus for every $k$ there is $h_k\in F\setminus K_k$ with $q(h_k)<1/k$.
Then $h_k\to0$, $C=\{0\}\cup\{h_k\}$ is compact in $\Ls$, and if
$C\subseteq K_m$, then $h_m\in K_m$, a contradiction.
\end{proof}

\begin{cor}\label{cor:qc}
If $\Ls$ is hemicompact, then $\Ls$ is quasi-complete.
\end{cor}

\begin{proof}
Apply Lemma~\ref{lem:seq} to a sequence $(K_n)$ with respect to which
$\Ls$ is hemicompact.
\end{proof}

The following lemma adapts the argument of K\c{a}kol and
Leiderman \cite[Theorem~2.7]{KL22}.

\begin{lemma}\label{lem:KL}
If $\Ls$ is $\sigma$-compact and quasi-complete, then $\Ls$ has a
countable fundamental system of bounded sets.
\end{lemma}

\begin{proof}
Let $L'=\bigcup_nK_n$ with compact $K_n$, and let $B_n$ be the closed
absolutely convex hull of $K_n$ in $\Ls$; it is bounded and, by
quasi-completeness, compact. Let $A\subseteq L'$ be a bounded set; we may
assume that it is closed and absolutely convex, so $A$ is compact. By
Fact~\ref{f:known}(7), $L'_A$ is a Banach space with the norm
$\|\cdot\|_A$, and the sets $nB_n\cap L'_A$ are closed in it, absolutely
convex and cover $L'_A$. By the Baire theorem there are $k$, $x\in L'_A$
and $\varepsilon>0$ with $x+\varepsilon A\subseteq kB_k$. By symmetry
$-x+\varepsilon A\subseteq kB_k$, and by convexity
$\varepsilon A\subseteq kB_k$. Put
$C_n=n\operatorname{conv}(B_1\cup\dots\cup B_n)$; these sets are bounded,
and $A\subseteq\frac k\varepsilon B_k\subseteq C_m$ for
$m\ge\max(k,k/\varepsilon)$. Thus $(C_n)$ is a fundamental system of
bounded sets.
\end{proof}

\begin{proof}[Proof of Proposition~\ref{prop:mackey}]
(1)$\Rightarrow$(2) is obvious. (2)$\Rightarrow$(3): the polars
$U_n^\circ$ of basic neighbourhoods of zero are absolutely convex and
compact (Fact~\ref{f:known}(1)) and cover $L'$. (4)$\Rightarrow$(1): the
topology $\tau(L,L')$ is the topology of uniform convergence on
absolutely convex compact subsets of $\Ls$, and by (4) it is given by the
countable family of seminorms $x\mapsto\sup_{f\in A_n}|f(x)|$.

(3)$\Rightarrow$(4). Let $L'=\bigcup_nA_n$ with absolutely convex compact
$A_n$, and let $A$ be an absolutely convex compact set. By
Fact~\ref{f:known}(7), $L'_A$ is a Banach space, and the sets
$nA_n\cap L'_A$ are closed in it, absolutely convex and cover $L'_A$. As
in the proof of Lemma~\ref{lem:KL}, the Baire theorem gives
$\varepsilon A\subseteq kA_k$ for some $\varepsilon>0$ and $k$. Put
$A'_n=n\operatorname{conv}(A_1\cup\dots\cup A_n)$; the convex hull of
finitely many absolutely convex compact sets is an absolutely convex
compact set. Then $A\subseteq\frac k\varepsilon A_k\subseteq A'_m$ for
$m\ge\max(k,k/\varepsilon)$, i.e., $(A'_n)$ is the required subfamily.
\end{proof}

\begin{proof}[Proof of Theorem~\ref{thm:main}]
(i)$\Rightarrow$(v): compact sets cover $L'$ (singletons), and $\Ls$ is
quasi-complete by Corollary~\ref{cor:qc}.

(v)$\Rightarrow$(iv) is Lemma~\ref{lem:KL}.

(iv)$\Rightarrow$(i): let $(B_n)$ be a fundamental system of bounded
sets. By quasi-completeness the closures $\overline{B_n}$ are compact, and
every compact set is bounded and hence contained in some $B_n$; so $\Ls$
is hemicompact with respect to $(\overline{B_n})$.

(iv)$\Leftrightarrow$(iii). The topology $\tau(L,L')$ is the topology of
uniform convergence on absolutely convex compact subsets of $\Ls$, and
$\beta(L,L')$ is that on bounded subsets of $\Ls$; always
$\tau(L,L')\subseteq\beta(L,L')$. If $\Ls$ is quasi-complete, then the
closed absolutely convex hull of a bounded set is compact, and
$\beta(L,L')=\tau(L,L')$. Conversely, let $\beta(L,L')=\tau(L,L')$ and let
$B\subseteq L'$ be bounded. Then $B_\circ$ is a neighbourhood of zero in
$\tau(L,L')$ and contains $C_\circ$ for some absolutely convex compact
set $C$; by the bipolar theorem $B\subseteq(B_\circ)^\circ\subseteq
(C_\circ)^\circ=C$, so $\Ls$ is quasi-complete. Finally, $\beta(L,L')$ is
metrizable if and only if $\Ls$ has a countable fundamental system of
bounded sets.

(iii)$\Leftrightarrow$(ii) is Fact~\ref{f:known}(3) applied to
$\tau(L,L')$.

(ii)$\Leftrightarrow$(vi) is Proposition~\ref{prop:mackey},
(1)$\Leftrightarrow$(3).
\end{proof}

\begin{proof}[Proof of Theorem~\ref{thm:metr}]
By Fact~\ref{f:known}(4), $\tau=\tau(L,L')$.

(a)$\Rightarrow$(b). The polars $U_n^\circ$ are compact
(Fact~\ref{f:known}(1)). A compact subset of $\Ls$ is bounded and, by
barrelledness, equicontinuous, i.e., it is contained in some
$U_n^\circ$.

(b)$\Rightarrow$(c) is obvious. (c)$\Rightarrow$(a): by
Theorem~\ref{thm:main} the space $(L,\tau(L,L'))=(L,\tau)$ is barrelled.

(a)$\Rightarrow$(d): bounded sets are equicontinuous and lie in the
compact sets $U_n^\circ$. (d)$\Rightarrow$(a): if $D$ is a barrel, then
$D^\circ$ is bounded and closed, hence compact and absolutely convex, and
$D=(D^\circ)_\circ$ is a neighbourhood of zero in $\tau(L,L')=\tau$ (this
is Fact~\ref{f:known}(6)).

(a)$\Rightarrow$(e): a convergent sequence is bounded.
(e)$\Rightarrow$(a) is Fact~\ref{f:known}(5).

By Fact~\ref{f:known}(9), the $\beta(L',L)$-bounded subsets of $L'$ are
precisely the equicontinuous ones. (a)$\Rightarrow$(f): compact sets are
bounded and, by barrelledness, equicontinuous. (f)$\Rightarrow$(e): a
null sequence together with zero is a compact set.
\end{proof}

\begin{remark}\label{rem:KLq}
It is noted in \cite[Remark~2.8]{KL22} that for metrizable $L$ every
cover of $L'$ by countably many absolutely convex compact subsets of
$\Ls$ consists of $\beta(L',L)$-bounded sets, and it is asked whether
absolute convexity can be dropped. Since $L'$ is always covered by the
compact sets $U_n^\circ$, adding an arbitrary compact set to them again
gives a countable cover; therefore, by
Theorem~\ref{thm:metr}((a)$\Leftrightarrow$(f)), the answer is affirmative
if and only if $L$ is barrelled.
\end{remark}

\begin{proof}[Proof of Theorem~\ref{thm:normed}]
Apply Theorem~\ref{thm:metr} to $U_n=\frac1nB_L$: then
$U_n^\circ=nB_{L'}$, and equicontinuity in $L'$ means norm boundedness.
This gives (a)$\Leftrightarrow$(b)$\Leftrightarrow$(c)$\Leftrightarrow$(e).
Further, (a)$\Rightarrow$(d): compact sets are bounded, and by
barrelledness bounded sets are equicontinuous; (d)$\Rightarrow$(e): if
$f_n\to0$, then $\{0\}\cup\{f_n\}$ is compact.

For normed $L$ the implication (e)$\Rightarrow$(a) is easily proved
directly. If $(f_n)$ is $\sigma(L',L)$-bounded but $\|f_n\|\ge n^2$, then
$g_n=f_n/n\to0$ in $\Ls$ and $\|g_n\|\ge n$, contrary to (e). Hence every
$\sigma(L',L)$-bounded set is norm bounded; for a barrel $D$ we have
$D^\circ\subseteq CB_{L'}$ and $D=(D^\circ)_\circ\supseteq C^{-1}B_L$.
\end{proof}

\begin{proof}[Proof of Theorem~\ref{thm:normable}]
(a)$\Rightarrow$(b). Let $U$ be the unit ball of a norm defining
$\tau(L,L')$. Then $U^\circ$ is an absolutely convex compact subset of
$\Ls$ (Fact~\ref{f:known}(1)) and $L'=\bigcup_n nU^\circ$, so $L'$ is the
linear span of $U^\circ$. Since $(L,\tau(L,L'))$ is a barrelled Mackey
space, $\Ls$ is quasi-complete by Fact~\ref{f:known}(6).

(b)$\Rightarrow$(a). Let $L'=\spn K$ with $K$ compact, and let $A$ be the
closed absolutely convex hull of $K$; by quasi-completeness $A$ is
compact, and $L'=\bigcup_n nA$. Let $B$ be an absolutely convex compact
subset of $\Ls$. By Fact~\ref{f:known}(7), $L'_B$ is a Banach space
covered by the closed absolutely convex sets $nA\cap L'_B$; as in the
proof of Lemma~\ref{lem:KL}, the Baire theorem gives $\varepsilon
B\subseteq kA$ for some $\varepsilon>0$ and $k$. Hence $\tau(L,L')$, the
topology of uniform convergence on absolutely convex compact subsets of
$\Ls$, is given by the single seminorm $x\mapsto\sup_{f\in A}|f(x)|$,
which is a norm since $A$ spans $L'$. Thus $(L,\tau(L,L'))$ is normable,
and it is barrelled by Fact~\ref{f:known}(6).

(b)$\Leftrightarrow$(c). This is \cite[Theorem~3.2]{KL22}; for the
reader's convenience we sketch the argument. If $L'=\spn K$ with $K$
compact, then $x\mapsto(f\mapsto f(x))|_K$ is a linear injection of $L$
into $C_p(K)$, and it is a homeomorphism of $(L,w)$ onto its image, since
$\sigma(L,L')$ is the topology of pointwise convergence on $K$ (because
$\spn K=L'$). Conversely, let $T\colon(L,w)\to C_p(Q)$ be a linear
embedding. The map $q\mapsto\delta_q\circ T$ from $Q$ into $\Ls$ is
continuous, so $K=\{\delta_q\circ T:q\in Q\}$ is compact. Every $f\in L'$
is a continuous functional on the subspace $T(L)$ of $C_p(Q)$, hence
extends to $C_p(Q)$ and is a finite linear combination of evaluations;
thus $L'=\spn K$.
\end{proof}

\section{Function spaces}\label{sec:func}

In this section we study the weak$^*$ duals of function spaces: of
$C_p(X)$, of $C_b(X)$ with the pointwise topology, of $C_k(X)$ and, more
generally, of the spaces $C_\lambda(X)$ with the topology of uniform
convergence on the members of a family $\lambda$ of functionally bounded
sets. Throughout, $X$ is a Tychonoff space. We begin with a general
consequence of Theorem~\ref{thm:main} that applies to $C_p(X)$.

\begin{proposition}\label{prop:feral}
For an LCS $L$ the absolutely convex compact subsets of $\Ls$ are
finite-dimensional if and only if $\tau(L,L')=\sigma(L,L')$ \textup(this
part is standard\textup). For such $L$
(e.g., $L=C_p(X)$, see Lemma~\ref{lem:Lpac}) the space $\Ls$ is
quasi-complete if and only if it is \emph{feral}, i.e., all its bounded
subsets are finite-dimensional, and
\begin{align*}
\Ls\ \text{is hemicompact}&\iff \dim L'\le\aleph_0\ \text{and}\ \Ls\
\text{is feral}\\
&\iff \sigma(L,L')\ \text{is metrizable and}\ L\ \text{is barrelled}.
\end{align*}
\end{proposition}

\begin{proof}
The topology $\tau(L,L')$ is the topology of uniform convergence on
absolutely convex compact subsets of $\Ls$; if they are
finite-dimensional, then $\tau(L,L')\subseteq\sigma(L,L')$. Conversely, if
$\tau(L,L')=\sigma(L,L')$, then an absolutely convex compact set $K$ is
equicontinuous with respect to $\sigma(L,L')$, i.e., it is contained in
the absolutely convex hull of a finite set.

If $\Ls$ is quasi-complete, then a closed absolutely convex bounded set is
compact and hence finite-dimensional; conversely, a feral TVS is
quasi-complete (Fact~\ref{f:known}(12)). The last assertion follows from
Theorem~\ref{thm:main}((i)$\Leftrightarrow$(ii)$\Leftrightarrow$(iv)), the
equality $\tau(L,L')=\sigma(L,L')$ and Fact~\ref{f:known}(8).
\end{proof}

For strong duals, the feral property has been studied in connection
with the distinguishedness of the spaces $C_p(X)$ and $C_k(X)$;
see~\cite{KS23}.

\begin{lemma}\label{lem:Lpac}
The space $C_p(X)$ is a Mackey space, i.e.,
$\tau(C(X),E(X))=\sigma(C(X),E(X))$, and the absolutely convex compact
subsets of $L_p(X)$ are finite-dimensional.
\end{lemma}

\begin{proof}
The first assertion is Fact~\ref{f:func}(1); the second follows from the
first by Proposition~\ref{prop:feral}.
\end{proof}

The space $L_p(X)$ is feral if and only if all functionally bounded
subsets of $X$ are finite (Fact~\ref{f:func}(4)).

\begin{proposition}\label{prop:Cp}
Let $L=C_p(X)$, so that $L'=E(X)$ and $\Ls=L_p(X)$. Then
\begin{enumerate}[label=\textup{(\arabic*)}]
\item $\tau(L,L')=\sigma(C(X),E(X))$ is metrizable $\iff$ $X$ is
countable (since $\dim E(X)=|X|$; Fact~\ref{f:known}(8));
\item $L_p(X)$ is quasi-complete $\iff$ $L_p(X)$ is feral $\iff$ $C_p(X)$
is barrelled $\iff$ all functionally bounded subsets of $X$ are finite
(Facts~\ref{f:known}(6) and~\ref{f:func}(2));
\item $L_p(X)$ is $\sigma$-compact $\iff$ $X$ is $\sigma$-compact
(Fact~\ref{f:func}(3));
\item $L_p(X)$ is hemicompact $\iff$ $X$ is countable and $L_p(X)$ is
feral $\iff$ $X$ is countable and all compact subsets of $X$ are finite
(Proposition~\ref{prop:feral}; for countable $X$ all functionally bounded
subsets are finite if and only if all compact subsets are finite, by
Fact~\ref{f:func}(6)).
\end{enumerate}
\end{proposition}

For example, $L_p(\mathbb N)=(\varphi,\sigma(\varphi,\R^{\mathbb N}))$ is
hemicompact; for a compact space $Q$ the space $L_p(Q)$ is always
$\sigma$-compact but hemicompact only for finite $Q$; if $Q$ is a
convergent sequence together with its limit, then the topology $\tau(L,L')$ is metrizable, but $C_p(Q)$ is not
barrelled; for $X=[0,1]$ the space $L_p(X)$ is $\sigma$-compact, but
$\tau(L,L')$ is not metrizable. Thus, on the spaces $L_p(X)$,
metrizability of $\tau(L,L')$ and quasi-completeness of $\Ls$ are
independent, $\sigma$-compactness is weaker than their conjunction, and
hemicompactness is exactly their conjunction.

The space $L_p(X)$ can also be hemicompact when $X$ is non-discrete. Let
$X=\omega\cup\{p\}$, where $p$ is a free ultrafilter on $\omega$, the
points of $\omega$ are isolated, and the neighbourhoods of $p$ are the
sets $A\cup\{p\}$, $A\in p$. The space $X$ is countable and all its
compact subsets are finite: if a compact set $K$ were infinite, then
$K\cap\omega$ could be split into two infinite sets, one of which does not
belong to $p$ and is therefore closed and discrete in $K$, which is
impossible. By Proposition~\ref{prop:Cp}, $L_p(X)$ is hemicompact, while
$C_p(X)$ is a metrizable barrelled space different from $\R^X$ (since
$X$ is not discrete).

The space $L_b(X)$ is $\Ls$ for $L=C_b(X)$ with the topology of pointwise
convergence.

\begin{proposition}\label{prop:Lb}
The space $L_b(X)$ is hemicompact if and only if $X$ is finite.
\end{proposition}

\begin{proof}
Let $X$ be infinite. The ball $U=\{f\in C_b(X):\|f\|_\infty\le1\}$ is a
barrel in $L$. If $U$ were a neighbourhood of zero in $\tau(L,L')$, then
its polar, the ball $A=\{\mu\in E(X):\|\mu\|_1\le1\}$ (indeed,
$\sup_{\|f\|_\infty\le1}|\sum_ia_if(x_i)|=\sum_i|a_i|$ for distinct
$x_i$), would be an absolutely convex compact subset of $L_b(X)$
(Fact~\ref{f:known}(1)). The Minkowski functional of $A$ on
$L'_A=E(X)$ is the norm $\|\cdot\|_1$, so by Fact~\ref{f:known}(7) the
space $(E(X),\|\cdot\|_1)$ would be a Banach space, which is false for
infinite $X$. Hence $(L,\tau(L,L'))$ is not barrelled, and
$L_b(X)$ is not hemicompact by Theorem~\ref{thm:main}.
\end{proof}

For the free spaces $L(X)$, $V(X)$ and the free groups $F(X)$, $A(X)$,
which are not weak duals, see~\cite{R-hf}.

Note also that $C_p(X)$ is always non-$S_\sigma$ \cite{SSR17}, so
$L_p(X)$ contains no infinite-dimensional metrizable linear subspaces;
moreover, all $\kappa$-Fr\'echet--Urysohn subgroups of $L_p(X)$ are
locally finite-dimensional \cite[Theorem~7.6]{LRS-TVS}.

\medskip
Let now $\lambda$ be a family of functionally bounded subsets of $X$
covering $X$ and closed under finite unions, and let $C_\lambda(X)$ be
the space $C(X)$ with the topology of uniform convergence on the members
of $\lambda$ \cite[Ch.~I, \S\,2]{McCoyNtantu}. For $\lambda$ consisting of the finite
sets we obtain $C_p(X)$, and for all compact sets, $C_k(X)$. For
$A\in\lambda$ put $\widehat A=\operatorname{cl}_{\upsilon X}A$; this is a
compact set, and $\sup_A|f|=\sup_{\widehat A}|f^\upsilon|$. Put
$Y_\lambda=\bigcup_{A\in\lambda}\widehat A$, so that $X\subseteq
Y_\lambda\subseteq\upsilon X$.

For $C_k(X)$ the following lemma is \cite[Theorem~2.12 (due to Schmets)
and Proposition~2.13]{KL22}.

\begin{lemma}\label{lem:Clambda}
Let $L=C_\lambda(X)$.
\begin{enumerate}[label=\textup{(\alph*)}]
\item Every $\xi\in L'$ has the form $\xi(f)=\int_{\widehat
A}f^\upsilon\,d\mu$ for some $A\in\lambda$ and a Radon measure $\mu$ on
$\widehat A$.
\item For $p\in\upsilon X$ the functional $\delta_p(f)=f^\upsilon(p)$ is
continuous on $L$ if and only if $p\in Y_\lambda$.
\item The map $p\mapsto\delta_p$ is a homeomorphism of $Y_\lambda$ (with
the topology induced from $\upsilon X$) onto a closed subset of $\Ls$.
\end{enumerate}
\end{lemma}

\begin{proof}
(a) For $\xi\in L'$ there are $A\in\lambda$ and $c>0$ with
$|\xi(f)|\le c\sup_{\widehat A}|f^\upsilon|$ for all $f\in C(X)$. The
restriction map $r\colon C(\upsilon X)\to C(\widehat A)$ is surjective (a
continuous function on the compact set $\widehat A$ extends to
$\beta\upsilon X$). If $r(f^\upsilon)=r(h^\upsilon)$, then
$|\xi(f)-\xi(h)|\le c\sup_{\widehat A}|f^\upsilon-h^\upsilon|=0$; hence
$\tilde\xi(r(f^\upsilon))=\xi(f)$ is a well-defined linear functional on
$C(\widehat A)$ with $|\tilde\xi(g)|\le c\|g\|_\infty$. By the Riesz
theorem $\tilde\xi$ is given by a Radon measure $\mu$ on $\widehat A$,
and $\xi(f)=\int_{\widehat A}f^\upsilon\,d\mu$.

(b) If $p\in\widehat A$, then $|f^\upsilon(p)|\le\sup_A|f|$. Conversely,
if $|f^\upsilon(p)|\le c\sup_{\widehat A}|f^\upsilon|$ and
$p\notin\widehat A$, then $g\in C(\upsilon X)$ with $g|_{\widehat A}=0$
and $g(p)=1$ gives a contradiction.

(c) The topology $\sigma(L',L)$ on $\{\delta_p:p\in Y_\lambda\}$ is the
weak topology generated by $C(\upsilon X)$, i.e., the subspace topology
from $\upsilon X$. Let $\delta_{p_i}\to\xi\in L'$. Then
$\xi(fg)=\xi(f)\xi(g)$ and $\xi(1)=1$; by (a), $\xi$ is a measure on some
$\widehat A$, and since the restriction map is surjective, it defines a
multiplicative functional on $C(\widehat A)$, i.e., $\xi=\delta_q$ with
$q\in\widehat A\subseteq Y_\lambda$.
\end{proof}

\begin{proposition}\label{prop:Clambda}
Let $L=C_\lambda(X)$.
\begin{enumerate}[label=\textup{(\alph*)}]
\item If $\Ls$ is hemicompact, then $Y_\lambda=\upsilon X$ and
$\upsilon X$ is hemicompact.
\item If $\upsilon X$ is hemicompact and every compact subset of
$\upsilon X$ is contained in some $\widehat A$, $A\in\lambda$, then
$C_\lambda(X)$ is linearly homeomorphic to $C_k(\upsilon X)$, it is
metrizable and barrelled, and $\Ls$ is hemicompact.
\end{enumerate}
\end{proposition}

\begin{proof}
(a) By Lemma~\ref{lem:Clambda}(c), $Y_\lambda$ is homeomorphic to a
closed subset of the hemicompact space $\Ls$ and is therefore
hemicompact. Hence $Y_\lambda$ is $\sigma$-compact, Lindel\"of and
therefore realcompact; since $X\subseteq Y_\lambda\subseteq\upsilon X$,
we get $Y_\lambda=\upsilon X$ (Fact~\ref{f:func}(6)).

(b) The topology of $C_\lambda(X)$ is the topology of uniform convergence
on the sets $\widehat A$, which are cofinal among the compact subsets of
$\upsilon X$, i.e., it is $C_k(\upsilon X)$. By the Arens theorem it is
metrizable; since $\upsilon X$ is realcompact, closed functionally bounded
subsets of $\upsilon X$ are compact, and $C_k(\upsilon X)$ is barrelled by
Fact~\ref{f:func}(5). It remains to apply Theorem~\ref{thm:metr}.
\end{proof}

\begin{cor}\label{cor:Ck}
The following are equivalent: \textup{(a)} $(C_k(X))'_\sigma$ is
hemicompact; \textup{(b)} $(C_k(X))'_\sigma$ is $\sigma$-compact;
\textup{(c)} $X$ is hemicompact, i.e., $C_k(X)$ is metrizable.
\end{cor}

\begin{proof}
(a)$\Rightarrow$(b) is trivial, and (b)$\Rightarrow$(c) is
\cite[Theorem~2.11 and Proposition~2.1]{KL22}. (c)$\Rightarrow$(a): for
$\lambda$ consisting of all compact sets, $\widehat A=A$; a hemicompact
$X$ is Lindel\"of, hence $X=\upsilon X$, and
Proposition~\ref{prop:Clambda}(b) applies. Proposition~\ref{prop:Clambda}(a)
also gives (a)$\Rightarrow$(c) directly.
\end{proof}

Thus, for $C_k(X)$ the condition is much weaker than for $C_p(X)$
(Proposition~\ref{prop:Cp}): for instance, $(C_k([0,1]))'_\sigma$ is
hemicompact while $L_p([0,1])$ is not; for $X=\mathbb N$ both spaces
coincide with $(\varphi,\sigma(\varphi,\R^{\mathbb N}))$; for
$X=\mathbb Q$ neither is hemicompact. For $C_p(X)$,
Proposition~\ref{prop:Clambda}(a) gives only the necessary condition
``$X$ is realcompact and hemicompact'', whereas the exact answer
(Proposition~\ref{prop:Cp}) is: $X$ is countable and all its compact
subsets are finite. For infinite $X$, a hemicompact $(C_k(X))'_\sigma$ is
not a $k$-space: by \cite{R-kd}, for a metrizable barrelled $L$ the space
$\Ls$ is a $k$-space only if $\dim L<\infty$.

The sufficient condition in Proposition~\ref{prop:Clambda}(b) is not
far from necessary in the extreme cases: for $\lambda$ consisting of all
compact sets it is also necessary (Corollary~\ref{cor:Ck}), and for
$\lambda$ consisting of finite sets and countable $X$ with only finite
compact subsets it holds as well, since then $X=\upsilon X$ is
hemicompact with respect to an increasing sequence of finite sets and
$C_p(X)=C_k(X)$.

\begin{question}\label{q:Clambda}
Let $\lambda$ consist of all functionally bounded subsets of $X$. Is it
true that $(C_\lambda(X))'_\sigma$ is hemicompact if and only if
$\upsilon X$ is hemicompact? What is the answer for an arbitrary
$\lambda$ in terms of $X$ and $\lambda$?
\end{question}

\section{Examples}

\begin{example}\label{ex:phi}
Let $\varphi=\bigoplus_{n}\R$ and $L=(\varphi,\sigma(\varphi,\varphi))$
with respect to the pairing $\langle x,y\rangle=\sum_ix_iy_i$. This
topology is induced from $\R^{\mathbb N}$ and is metrizable, i.e., $L$ is
the $\sigma$-product $\{x\in\R^{\mathbb N}:\ x_i\ne0$ for finitely many
$i\}$ in $\R^{\mathbb N}$; the same holds for $\Ls$. The absolutely convex
compact subsets of $(\varphi,\sigma(\varphi,\varphi))$ are
finite-dimensional (by the Baire theorem, since $\varphi=\bigcup_n\R^n$
with closed $\R^n$), hence $\tau(L,L')=\sigma(\varphi,\varphi)$ is
metrizable. However, $\Ls=(\varphi,\sigma(\varphi,\varphi))$ is not
hemicompact: the set $\{ne_n\}$ is bounded but not equicontinuous, i.e.,
$L$ is not barrelled (cf.\ Fact~\ref{f:known}(10)).
\end{example}

\begin{example}\label{ex:normed}
The normed space $L=(\varphi,\|\cdot\|_\infty)$ has countable dimension
and is not barrelled by Fact~\ref{f:known}(10); by
Theorem~\ref{thm:normed}, the space $\Ls=(\ell_1,\sigma(\ell_1,\varphi))$
is not hemicompact. Another example of a non-barrelled normed space is
$T(\ell_2)$ with the norm of $\ell_2$, where $T=\operatorname{diag}(1/k)$:
the set $T(B_{\ell_2})$ is absolutely convex and absorbs $T(\ell_2)$; it
is compact in $\ell_2$ and hence closed in $T(\ell_2)$, i.e., it is a
barrel; being compact in an infinite-dimensional normed space, it is not
a neighbourhood of zero.
\end{example}

\begin{example}\label{ex:prod}
\begin{enumerate}[label=\textup{(\arabic*)}]
\item Let $L=\prod_nE_n$, where the $E_n$ are normed. Then $L$ is
metrizable, $\Ls=\bigoplus_n(E_n)'_{w^*}$, and by Theorem~\ref{thm:metr}
and Fact~\ref{f:known}(11): $\Ls$ is hemicompact $\iff$ $L$ is barrelled
$\iff$ all $E_n$ are barrelled.
\item Let $L\subseteq\R^{\mathbb N}$ be a dense subspace. Then
$L'=\varphi$, and the equicontinuous sets are the bounded subsets of the
spaces $\R^n=\spn\{e_1,\dots,e_n\}$ (the polars of the neighbourhoods
$\{x\in L:|x_i|\le\varepsilon,\ i\le n\}$). Therefore
$(\varphi,\sigma(\varphi,L))$ is hemicompact $\iff$ $L$ is barrelled
$\iff$ all $\sigma(\varphi,L)$-bounded sets are finite-dimensional.
Indeed, the first equivalence is Theorem~\ref{thm:metr}. If $L$ is
barrelled, then bounded sets are equicontinuous and hence
finite-dimensional. Conversely, a finite-dimensional bounded set lies in
some $\R^n$, on which $\sigma(\varphi,L)$ is the Euclidean topology
(since the restrictions of elements of $L$ to $\R^n$ form an
$n$-dimensional space, $L$ being dense in $\R^{\mathbb N}$), and is
therefore equicontinuous; since polars of barrels are bounded
(Fact~\ref{f:known}(2)), every barrel is a neighbourhood of zero.
\end{enumerate}
\end{example}

\end{document}